\documentclass[a4paper]{amsart}

\usepackage[utf8]{inputenc}	
\usepackage[T1]{fontenc}	
\usepackage{ae,aecompl,aeguill}

\title[Symmetric functions and the rational Steenrod
  Algebra]{Deformation of symmetric functions and the
  rational Steenrod algebra}
\author{Florent Hivert}
\address{Institut Gaspard Monge, Université de Marne-la-Vallée, France}
\email{Florent.Hivert@univ-mlv.fr}
\author{Nicolas M. Thiéry}
\address{Laboratoire de Probabilités, Combinatoire et Statistiques,
  Université Claude Bernard Lyon I, France}
\email{Nicolas.Thiery@u-psud.fr}

\usepackage{array}
\usepackage[backref]{hyperref}
\usepackage{amsfonts}
\usepackage{amsmath}
\usepackage{xy}
\usepackage[matrix]{xypic}
\usepackage{verbatim}

\newtheorem{thm}{Theorem}[section]
\newtheorem{lem}{Lemma}[section]
\newtheorem{prop}{Proposition} 
\newtheorem{cor}{Corollary} 
\theoremstyle{definition}
\newtheorem{defn}{Definition}

\newtheorem{problem}{Problem}
\newtheorem{conjecture}{Conjecture}

\theoremstyle{remark}
\newtheorem{rem}{Remark}
\newtheorem{rems}{Remarks}

\newcommand{\scalar}[2]{\langle#1\,|\,#2\rangle}
\newcommand{\K}{\mathbb{K}}
\newcommand{\Z}{\mathbb{Z}}
\newcommand{\R}{\mathbb{R}}
\newcommand{\C}{\mathbb{C}}
\newcommand{\N}{\mathbb{N}}
\newcommand{\Q}{\mathbb{Q}}
\makeatletter
\newcommand{\indexP}[2][q]{_{\def\tmp{#1}\ifx\tmp\@empty\else#1,\fi#2}}
\newcommand{\genP}[1]{#1\indexP}
\renewcommand{\P}{\genP{P}}
\newcommand{\D}{\genP{P^*}}
\newcommand{\tP}{\genP{\widetilde P}}
\makeatother
\newcommand{\sg}{\mathfrak{S}}
\newcommand{\sym}{\operatorname{Sym}}
\newcommand{\steenrod}{\mathcal{A}}
\newcommand{\weyl}{\operatorname{Weyl}}
\newcommand{\steen}{\operatorname{Steen}}
\newcommand{\qsteen}[1][q]{\operatorname{Steen_{#1}}}
\newcommand{\tqsteen}[1][q]{\operatorname{TSteen_{#1}}}
\newcommand{\hit}{\operatorname{Hit}}
\newcommand{\qhit}[1][q]{\operatorname{Hit_{#1}}}
\newcommand{\tqhit}[1][q]{\operatorname{THit_{#1}}}
\newcommand{\harm}{\operatorname{Harm}}
\newcommand{\qharm}[1][q]{\operatorname{Harm_{#1}}}
\newcommand{\tqharm}[1][q]{\operatorname{THarm_{#1}}}
\newcommand{\p}{\partial}
\newcommand{\isom}{\simeq}
\newcommand{\hilb}[1][t]{\operatorname{Hilb}_{#1}}

\newcommand{\rank}{\operatorname{Rank}}
\newcommand{\coeff}{\operatorname{Coeff}}
\newcommand{\length}{\ell}
\newcommand{\tensor}{\otimes}
\newcommand{\ord}{\operatorname{ord}}
\newcommand{\os}{\operatorname{o}}
\newcommand{\specht}{\operatorname{Sp}}
\newcommand{\partof}{\vdash}
\newcommand{\strng}{\operatorname{String}}
\newcommand{\schub}{\operatorname{S}}

\begingroup
\catcode`@=11\relax%
\catcode`{=12\relax\catcode`}=12\relax%
\catcode`(=1\relax \catcode`)=2\relax%
\gdef\includeversion#1(%
  \expandafter\gdef\csname #1\endcsname()%
  \expandafter\gdef\csname end#1\endcsname()%
  \expandafter\gdef\csname ifver#1\endcsname(\@iden)%
)%
\gdef\excludeversion#1(%
  \expandafter\gdef\csname #1\endcsname%
    (\@bsphack\catcode`{=12\relax\catcode`}=12\relax\csname #1@NOTE\endcsname)%
  \long\expandafter\gdef\csname #1@NOTE\endcsname ##1\end{#1}%
    (\csname #1END@NOTE\endcsname)%
  \expandafter\gdef\csname #1END@NOTE\endcsname%
    (\@esphack\end(#1))%
  \expandafter\gdef\csname ifver#1\endcsname%
    (\@gobble)%
)%
\gdef\includever#1(\csname include#1\endcsname)%
\endgroup
\excludeversion{private}

\newcommand{\commentaire}[1]{}

\def\SetTableau#1#2#3#4{%
  \gdef\Tabvrule{\vrule\vrule width-0.4pt}
  \gdef\Tabhrule{\hrule\hrule height-0.4pt}
  \gdef\Tabstrut{\vrule height#1 depth#2 width0pt\relax}
  \gdef\Tabbox##1{\hbox to #3{\hskip0.4pt\hfill\Tabstrut$#4##1$\hfill}}
} 
\def\NormalTableau{\SetTableau{2.25ex}{0.75ex}{3ex}{}}
\NormalTableau
\def\Case#1{\vcenter{\Tabhrule%
                   \hbox{\Tabvrule\Tabbox{#1}\Tabvrule}\Tabhrule}}
\def\GenTab#1{\vcenter{\halign{&$\Case{##}$\cr#1}}\egroup}
\def\Tableau{%
  \bgroup%
  \let\ =\omit%
  \let\\=\cr%
  \offinterlineskip\GenTab}
\begin{document}

\begin{abstract}
  In 1999, Reg Wood conjectured that the quotient of
  $\Q[x_1,\dots,x_n]$ by the action of the rational Steenrod algebra
  is a graded regular representation of the symmetric group $\sg_n$.
  As pointed out by Reg Wood, the analog of this statement is a well
  known result when the rational Steenrod algebra is replaced by the
  ring of symmetric functions; actually, much more is known about the
  structure of the quotient in this case.

  We introduce a non-commutative $q$-deformation of the ring of
  symmetric functions, which specializes at $q=1$ to the rational
  Steenrod algebra. We use this formalism to obtain some partial
  results.  Finally, we describe several conjectures based on an
  extensive computer exploration. In particular, we extend Reg Wood's
  conjecture to $q$ formal and to any $q\in \C$ not of the form
  $-a/b$, with $a\in\{1,\dots,n\}$ and $b\in\N$.
\end{abstract}

\maketitle

\tableofcontents

\section{Introduction}

The \emph{rational Steenrod algebra} is the subalgebra $\steenrod$ of
the Weyl algebra generated by the \emph{rational Steenrod squares}
$D_k:=\sum_i x_i^{k+1}\p_i$~\cite{Wood.DOSA.1997}. The Steenrod
squares are derivations, and satisfy the Lie relations
$[D_k,D_l]=(l-k)D_{k+l}$. In particular, the Steenrod algebra is
generated by just $D_1$ and $D_2$.

Let $\C[X_n]:=\C[x_1,\dots,x_n]$ be a ring of polynomials, and set
$\pi:=x_1\dots x_n$.  Consider the action of the Steenrod algebra on
$\C[X_n]$ through the algebraic Thom map, so that for $f\in \steenrod$
and $p\in\C[X_n]$,
\begin{equation}
  \rho(f).p:=\frac{1}{\pi}f\pi.p.
\end{equation}
This action seems to share properties with the natural action of the
ring of symmetric polynomials on polynomials by multiplication.
\begin{conjecture}[Rational hit conjecture of
  Reg Wood~\cite{Wood.PSA.1998,Wood.HPSA.2001}] The quotient
  $\C[X_n]_{/\steenrod^+\C[X_n]}$ is a graded regular representation
  of the symmetric group $\sg_n$.
\end{conjecture}
In this article, we present preliminary results from research in
progress around this conjecture. We refer
to~\cite{Wood.DOSA.1997,Wood.PSA.1998,Wood.HPSA.2001} for motivations.

It is actually easier to get rid of the algebraic Thom map, by
conjugating once for all the Steenrod algebra by this map. So, we
consider instead the isomorphic algebra denoted $\steen$ generated by
the Weyl operators
\begin{equation}
  \P[]k:=\frac{1}{\pi}D_k\pi=\sum_i x_i^k (1+x_i\p_i).
\end{equation}
Note that those operators are no longer derivations.

Now, $\P[]k$ appears clearly as a usual symmetric power-sum, with a
deformation.  Since the goal is to transfer properties of the ring of
symmetric polynomials to the Steenrod algebra, this suggests that we
interpolate continuously between the two. Hence, we introduce the
\emph{$q$-Steenrod algebra} denoted $\qsteen$ which is generated by
the Weyl operators
\begin{equation}
  \P k:=\sum_i x_i^k (1+qx_i\p_i).
\end{equation}
The $0$-Steenrod algebra is the ring of symmetric polynomials, while
the $1$-Steenrod algebra is the rational Steenrod algebra conjugated
by the algebraic Thom map.

In the following, $q$ can be either a given complex number $q_0$, or a
formal parameter. In the later case we take $\K:=\C(q)$ as base field,
or $\K:=\C[q]$, when we just need a base ring. Roughly speaking, all
the statements we consider are of algebraic nature, and hold for $q$
formal if, and only if, they hold for $q$ generic (that is in most
cases for all but a countable set).

\subsection{Summary of the results}

In~\cite{Garsia_Haiman.OHGR}, a concrete realization $\harm$ of the
quotient $\C[X_n]_{/\sym^+\C[X_n]}$ is constructed as a space of
so-called \emph{harmonics}. Following a similar approach, we construct
a concrete realization $\qharm$ of the quotient
$\C[X_n]_{/\qsteen^+\C[X_n]}$.

Both our theoretical results, and our computer exploration leads us to
extend Reg Wood's conjecture as follows:
\begin{conjecture}
  \label{conj.main}
  Let $q$ be formal, or a complex number not of the form $-a/b$, with
  $a\in\{1,\dots,n\}$, $b\in\N$ and $a\leq b$. Then, for any $n$,
  $\qharm$ is isomorphic to $\harm$ as a graded $\sg_n$-module. In
  particular, it is isomorphic to the regular representation of
  $\sg_n$.
\end{conjecture}

The main results of this article are summarized in the following two
propositions. Recall that a monomial $x^K=x_1^{k_1}\cdots x_n^{k_n}$
in $\K[X_n]$ is called \emph{staircase monomial} if $k_i\leq n-i$, for
all $i$.
\begin{prop}
  \label{prop.goingup}
  The following are equivalent:
  \begin{itemize}
  \item[(a)] Conjecture~\ref{conj.main} holds for $q$ formal;
  \item[(b)] For any $n$, $\dim\qharm\geq \dim\harm$;
  \item[(c1)] For any $n$, and any polynomial $p\in\qharm$, all the
    variables have degree at most $n-1$ in $p$;\footnote{2007/06: as
      noted by Adriano Garsia, this statement is in fact incorrect}
  \item[(c2)] For any $n$, the staircase monomials are the leading
    terms of the polynomials in $\qharm$ for the lexicographic order.
  \item[(d)] Any basis of $\C[X_n]_{/\sym^+\C[X_n]}$ is a basis of
    $\K[X_n]_{/\qsteen^+\K[X_n]}$.
  \end{itemize}
\end{prop}

\begin{prop}
  \label{prop.goingdown}
  Assume that conjecture~\ref{conj.main} holds for $q$ generic. Then,
  for $q_0\in\C$, the following are equivalent:
  \begin{itemize}
  \item[(a)] Conjecture~\ref{conj.main} holds for $q_0$;
  \item[(b)] For any $n$, $\dim\qharm[q_0]\leq \dim\harm$;
  \item[(c)] For any $n$, the non-staircase monomials are the leading
    terms of the polynomials in $\qhit$ for the lexicographic order.
  \end{itemize}
\end{prop}

\begin{private}
Comparison of the dimensions of $\qharm$ and $\harm$:

\begin{tabular}{|l|l|l|l|l|}
  \hline
	  		& $q$ formal	& $q=-a/b$		& $q=0$	& $q$ other
    \\\hline
  $n\geq d+1$	 	& $=$ 		& $\geq$ (conj: $=$)	& $=$	& $\geq$ (conj: $=$) \\\hline
  $\qsteen$ truncated 	& $=$ 		& $\geq$, can be $>$	& $=$	& $\geq$ (conj: $=$) \\\hline
  $\qsteen$		& $<=$ (conj: $=$)
				& ? (can be $>$, conj: $\geq$)	& $=$	& ? (conj: $=$)
		\\\hline
\end{tabular}
\end{private}

\subsection{Organization of the paper}

After a background section, we define the $q$-deformed Steenrod
algebra $\qsteen$ (Definition \ref{defn.qsteen}), together with
$q$-hit and $q$-harmonic polynomials. We compute the dimension and the
structure of $\qsteen$ (Theorems \ref{thm.structSteen} and
\ref{thm.isomSteen}).

In the section~\ref{sec.Specializations}, we describe the relations
between the generic setting (when $q$ is a formal parameter), the
complex setting (when $q$ in a complex number) and the classical
setting (when $q=0$).

Then, we analyze the space of polynomials that can be hit from
harmonics by successive application of at most $n$ Steenrod operators.
This provides information in small degrees. In particular we show that
the submodules of degree smaller than $n$ of $\K[X_n]/\sym^+$ and
$\K[X_n]/\qsteen^+$ are isomorphic (Corollary~\ref{cor.truncated}).

In the next section, we investigate the relation between the
$q$-harmonics in $n$ variables and in $n+1$ variables. We construct a
trick to go from $n$ to $n+1$: the so-called \emph{strings}. The main
result of this part is that the conjecture is equivalent to the fact
that the harmonics in $n$-variables are of degree smaller than $n$ in
each variable (Proposition~\ref{prop.equivShortString}). We also
analyze completely the simple cases of $1$ and $2$ variables.

In the next-to-last section we describe some attempts to prove
conjecture~\ref{conj.main} by looking for $q$-analogs of elementary
symmetric function (Subsection~\ref{subsection.Elementary}) and
$q$-analogs of Schubert polynomials
(Subsection~\ref{subsection.Schubert}), and we explain why they fail.

Finally, the last-section is devoted to a computer exploration which
motivates our conjectures.  In particular, we give, for small numbers
of variables and small degrees, a tight superset of the values of $q$
for which conjecture~\ref{conj.main} does not hold.  Most computations
where realized with the open source library
\texttt{MuPAD-Combinat}~\cite{MuPAD-Combinat}\footnote{\url{http://mupad-combinat.sf.net}}
for the Computer Algebra System \texttt{MuPAD}~\cite{MuPAD.96}. The
extra tools we developed at this occasion are available upon request,
and will eventually be integrated into \texttt{MuPAD-Combinat}, after
appropriate cleanup and documentation.

\section{Preliminaries}
\commentaire{Notation? recalls ? Background?}

The main topic of this paper is the study of the $q$-deformation of a
module. Hence we will work in two different settings: either $q:=q_0$
is a complex number, in which case the base field is $\K:=\C$, or $q$
is a formal complex number, and $\K:=\C(q)$. Many of the results are
true in both settings; in this case, we use $\K$ as base field without
precision.

For an alphabet $X$ of commuting variables, we denote respectively by
$\K[X]$, $\K(X)$, and $\sym(X)$ the ring of polynomials, the field of
fractions, and the ring of symmetric functions in the variables in
$X$. We denote by $X_n$ the alphabet of the $n$ commuting variables
$\{x_1,x_2,\dots,x_n\}$. Unless explicitly stated, $X:=X_\infty$ is
the alphabet in the countably many variables
$\{x_1,\dots,x_n,\dots\}$, and we use the shorthand
$\sym=\sym(X_\infty)$.

\subsection{Graded rings and Hilbert series}

In the following, most of the vector spaces we consider are graded,
with finite dimensional homogeneous components, and comparing their
dimensions will be one of our main tools. The \emph{Hilbert series} of
a graded vector space $A:=\bigoplus_{d=0}^\infty A_d$ is the
generating series of the dimensions of the homogeneous components of
$A$:
\begin{equation}
  \hilb(A) := \sum_{d=0}^\infty \dim A_d t^d.
\end{equation}
Given two formal series $H(t):=\sum_d h_d t^d$ and $L(t):=\sum_d l_d
t^d$, we say that $H(t)$ is \emph{dominated} by $L(t)$, and write
$H(t)\leq L(t)$, whenever $h_d\leq l_d$ for all $d$. The relation of
domination is preserved by multiplication by series with nonnegative
coefficients, such as Hilbert series and generating functions.

Let us recall some basic Hilbert Series.
\begin{prop}
  The Hilbert series of $\K[X_n]$ is
  \begin{equation}
    \label{eq.HilbPoly}
    \hilb(\K[X_n]) = \frac{1}{(1-t)^n}\,.
  \end{equation}
  
  The Hilbert series of the algebra $\sym := \sym(X_\infty)$ of symmetric
  functions on a countably many variables is
  \begin{equation}
    \label{eq.HilbSym}
    \hilb(\sym) = \prod_{i>0} \frac{1}{1-t^i}\,.
  \end{equation}
  
  The Hilbert series of the algebra $\sym(X_n)$ of symmetric
  polynomials on $n$ variables is
  \begin{equation}
    \label{eq.HilbSymn}
    \hilb(\sym(X_n)) = \prod_{i=1}^{n} \frac{1}{1-t^i}\,.
  \end{equation}
\end{prop}

\subsection{Representations of the symmetric group on polynomials}

The representation theory of the symmetric groups is a well known
combinatorial subject. The goal of this subsection is to recall
several basic useful facts.  For more details, the reader can refer
to~\cite{Fulton_Harris.RT,Garsia_Haiman.OHGR,Sagan.SG} or, in French,
to~\cite{Krob.EC.1995}.

The irreducible representations $V_\lambda$ of the symmetric group
$\sg_n$ are naturally indexed by partitions
$\lambda:=(\lambda_1\geq\lambda_2\geq\dots\geq\lambda_k>0)$ of sum $n$
(denoted by $\lambda\partof n$).  There is a particular occurrence of
$V_\lambda$, called the Specht module, in the natural representation
of $\sg_n$ on polynomials. This module is spanned by the so-called
Specht polynomials. Recall that the partition $\lambda$ is depicted
(in the French way) as a diagram of boxes putting $\lambda_i$ boxes in
the $i$-th row, starting from the bottom. A filling of the boxes by
the numbers $(1,2\dots,n)$ is called a \emph{standard filling} $F$ of
shape $\lambda$. A standard filling which is increasing from left to
right along rows and from bottom to top along columns is called a
\emph{standard tableau}.

For example here is the diagram of the partition $\lambda:=(5,3,2)$
together with a standard filling and a standard tableau of shape
$\lambda$.
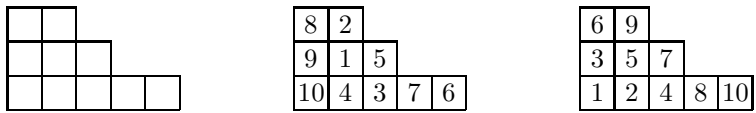
\[
\Tableau{ & \\ & & \\ & & & & \\}
\hskip1.5cm
\Tableau{8 & 2 \\ 9 & 1 & 5 \\ 10 & 4 & 3 & 7 & 6 \\}
\hskip1.5cm
\Tableau{6 & 9 \\ 3 & 5 & 7 \\ 1 & 2 & 4 & 8 & 10 \\}
\]
\medskip

Recall that the \emph{Vandermonde determinant} of a finite alphabet of
variables $\{y_1,\dots,y_k\}$ is defined by
\begin{equation}
  \label{eq.DefVandermonde}
  \Delta(y_1,\dots,y_k) :=
  \begin{vmatrix}
    1   & y_1 & y^2_1 & \dots  & y^{k-1}_1 \\
    1   & y_2 & y^2_2 & \dots  & y^{k-1}_2 \\
 \vdots & \vdots  &  \vdots   & \ddots & \vdots        \\
    1   & y_k & y^2_k & \dots  & y^{k-1}_k \\
  \end{vmatrix}
  = \prod_{1\leq i < j \leq k} (y_j - y_i).
\end{equation}
For each standard filling $F$, the \emph{Specht polynomial}
$\specht_F$ is defined as the product of the Vandermonde determinants
of the variables indexed by the columns of $F$.

For example the Specht polynomial associated with the previous filling
is
\begin{equation*}
  \begin{split}
    \specht_F &:=
    \Delta(x_{10},x_9,x_8) \cdot \Delta(x_4,x_1,x_2) \cdot
    \Delta(x_3,x_5) \cdot \Delta(x_7) \cdot\Delta(x_6) \\
    & =
    (x_9 - x_{10})(x_8 - x_{10})(x_8 - x_9) \cdot
    (x_1 - x_4) (x_2 - x_4) (x_2 - x_1) \cdot
    (x_5 - x_3) \cdot 1 \cdot 1\,.
  \end{split}
\end{equation*}
It is easy to see that, for $\sigma\in\sg_n$, and $F$ a filling,
$\sigma \specht_F = \specht_{\sigma F}$. Hence, the span of
$\{\specht_F\}$, where $F$ runs through the standard fillings of a
given shape $\lambda\partof n$ is a representation $V_\lambda$ of
$\sg_n$, which is called \emph{Specht module} indexed by $\lambda$.
\begin{prop}
  The Specht modules $(V_\lambda)_{\lambda\partof n}$ form a complete
  family of irreducible representations of $\sg_n$.
\end{prop}
The Specht polynomials indexed by fillings are not linearly
independent. A basis of the Specht module $V_\lambda$ is given by
$\{\specht_F\}$, where $F$ runs through the standard tableaux of shape
$\lambda$.

Note that the polynomials $\specht_F$ are homogeneous. Furthermore,
each Specht module is the lowest degree occurrence of the
corresponding irreducible representation; more precisely:
\begin{prop}
  Let $d$ be the degree of the Specht module $V_\lambda$.  The
  multiplicity of the irreducible representation $V_\lambda$ in
  $\K[X_n]$ in degree $d$ is $1$;  Any occurrence of a
  representation isomorphic to $V_\lambda$ in $\K[X_n]$ occurs in
  degree at least $d$.
\end{prop}
There is a construction for higher degree occurrences of
representations isomorphic to $V_\lambda$ by
\cite{Terasoma_Yamada.1993} using so-called \emph{higher Specht
  polynomials}.  They form an explicit basis for the quotient module
$\K[X_n]_{/\sym^+(X_n)\K[X_n]}$.

\medskip

Let $W$ a representation of $\sg_n$. Recall that the direct sum of all
representations isomorphic to some $V_\lambda$ is a sub-representation
called the \emph{isotypic component} indexed by $\lambda$ of $W$ and
denoted $W_\lambda$. Then $W$ decomposes canonically as
$W=\bigoplus_{\lambda\partof n} W_\lambda$. Moreover, in the group
ring $\K\sg_n$ of $\sg_n$, there are special elements $e_\lambda$
called \emph{central orthogonal idempotents} such that $e_\lambda W =
W_\lambda$ and $e_ie_j= \delta_{i,j}$.

Consider a non zero polynomial $f$ in $V_\lambda$ and a representation
$W$ isomorphic to $V_\lambda$. By Maschke's theorem, the set of the
images of $f$ by all $\sym_n$-morphisms from $V_\lambda$ to $W$ is a
line. Moreover, for two linearly independent polynomials $f_1$ and
$f_2$, these lines are in direct sum.
\begin{defn}
  Let $f\in V_\lambda$ be a non zero polynomial, and $W$ be a
  $\sg_n$-representation. The set of the images of $f$ by all
  $\sg_n$-morphisms from $V_\lambda$ to $W$ is called the \emph{Garnir
    component} of $W$ associated to $f$ denoted by $W_\lambda(f)$.
\end{defn}
$W_\lambda(f)$ has naturally a structure of vector space isomorphic to
$\hom(V_\lambda, W)$ of which it is a sort of concrete realization.
Then we know that
\begin{prop}
  Let $W$ be a representation of $\sg_n$. Then
  \begin{equation}
    \label{eq.decompGarnir}
    W = \bigoplus_{\lambda\partof n} W_\lambda
      = \bigoplus_{\lambda\partof n}
        \left(\bigoplus_{\text{Shape}(T)=\lambda} W_\lambda(\specht_T)\right)\,,
  \end{equation}
  where the last sum is on standard tableaux of shape $\lambda$.
\end{prop}
Note that this direct sum decomposition is dependent of the chosen
basis of $V_\lambda$ (here the Specht basis).

\subsection{The Weyl algebra}

The following brief introduction to the Weyl algebra is essentially
taken from~\cite{Wood.DOSA.1997}. For a variable $x_i$ define the
partial derivative
$$
\partial_i:=\frac{\partial}{\partial x_i}.
$$
Following standard practice, we write abbreviated expressions for
monomials\commentaire{$x^K$ or $X^K$?}:
$$
x^K := \prod x_i^{k_i}, \quad \p^L := \p_i^{l_i},
$$
where $K:=(k_1,\dots,k_n,\dots)$ and $L:=(l_1,\dots,l_n,\dots)$ are
exponent vectors. The \emph{degree} of $x^K$ is $\deg(x^K):=\sum_i
k_i$, and the \emph{order} of $\p^L$ is $\ord(x^L):=\sum_i l_i$.

\begin{private}
  Notations, ... should we include them ?
Weyl operators: $f,g,h$

Polynomials: $p,q,r$. Depending on the context, a polynomial $p$ is
seen as a Weyl operator or a pure polynomial.

$fg$, $pq$: product in the Weyl algebra / product of polynomials.

$f.p$: action of the differential operator $f$ on the polynom $p$.

$.$ is not associative, and takes precedence over multiplication:
$fg.pq=(fg).(pq)$.

Remark: if $f$ is in standard form, $f.1$ extracts the monomials of
$f$ that do not contain any partial derivatives. We call $f.1$ the
purely polynomial part of $f$.

\begin{rems}[Some trivial computation rules]
  \begin{align}
    p.q &= pq\,,\\
    p.1 &= p\,,\\
    (f)_1 &= (f.1)_1\,,\\
    f.g.p &= fg.p\,,\\
    (fp - f.p).1 &= 0\,,\\
    \frac{1}{\pi}x_i\p_i\pi &= \p_ix_i\,,\\
    (fg).1&=(f(g.1)).1\,.
  \end{align}
\end{rems}
(the last one can probably expressed in a better way).
\end{private}

Weyl monomials are monomials in the variables $x_i$ and the partial
derivatives $\p_i$. They act by multiplication and derivation on
$\K[X_n]$, and form an algebra under addition and composition. For
example,
\begin{gather*}
  \partial_1 . x_1^4 = 4 x_1^3\,, \qquad
  \partial_1 . x_2^2 = 0\,, \qquad
  \partial_1\partial_2.p = \partial_2\partial_1 . p\,,\\
  x_1x_2 . (1+x_2x_4) = x_2x_1 (1+x_2x_4) = x_1x_2 + 2x_1x_2^2x_4\,,\\
  \partial_1 x_1 . p = \partial_1 . (x_1p) = (\partial_1.x_1) p + x_1(\partial_1.p) = (1+x_1\partial_1) . p\,,
\end{gather*}
This is formulated more precisely as follows:
\begin{defn}
  The Weyl algebra $\weyl:=\weyl(X)$ is the associative algebra, with
  unit, generated by $x_i,\partial_i$, subject to the relations
  $$
  [x_i,x_j]=0, \qquad [\p_i, \p_j]=0, \qquad [x_i,\p_j]=\delta_{ij},
  $$
  where square brackets denotes the Lie product and $\delta_{ij}$
  the Kronecker symbol.
  
  The Weyl algebras $\weyl(X_n)$ are defined for each $n$ in a similar
  way by restricting to the finite set of variables $x_1,\dots,x_n$
  and the corresponding partial derivatives.
\end{defn}

By repeated differentiation, any element of the Weyl algebra can be
expressed as a linear combination of Weyl monomials where the
polynomial part is on the left and the derivatives on the right.
For example,
\begin{align}
  \p_1 x_1 &= x_1 \p_1 + 1 \\
  \p_1^k x_1 &= x_1\p_1^k + k \p_1^{k-1} \\
  \p_1\p_2 x_1x_2 &= x_1 x_2 \p_1 \p_2 + x_1 \p_1 + x_2 \p_2 + 1
\end{align}
We shall refer to such monomials as the \emph{standard monomials}. As
an Abelian group, $\weyl$ is freely generated by the standard
monomials $x^K\p^L$, as $K$ and $L$ range over exponent vectors.

An element $f$ in the Weyl algebra lies in \emph{filtration} $k$ if
the maximum order of a term of $f$ in standard form does not exceed
$k$. Then, $\weyl(X_n)$ and $\weyl$ are filtered algebras in the sense
that the composition product of two elements, in filtration $m$ and
$k$ respectively lies in filtration $m+k$. The Weyl algebras are
graded by assigning grading $\deg(x^K)-\ord(\p^L)$ to $x^K\p^L$.

Given a polynomial or a Weyl operator $p$, we denote by $\os(p)$ the
\emph{orbit sum} of $p$, that is the sum of all the polynomials in the
orbit of $m$ under the action of the symmetric group. Typically, if
$\lambda$ is a partition, then $m_\lambda:=\os(x^\lambda)$ is the
\emph{monomial symmetric function} indexed by $\lambda$. Note that we
have to be a little bit careful, since $\os(p)$ might be undefined for
an infinite alphabet $X$, as for example for $\os(x_1+x_2)$; however,
for a monomial or Weyl monomial $m$, $\os(m)$ is always defined. We
will consider differential operators such as
\begin{equation}
  \P[]k:=\sum_i x_i^k (1+x_i\p_i):=\os(x_1^k (1+x_1\p_1))
\end{equation}
which are obtained by such symmetrization over an infinite alphabet.
To be precise, such operators do not belong to the Weyl algebra,
because they involve infinite sums. They do not even necessarily act
properly on $\K[X]$; for example $\P[]k.1=x_1^k+\dots+x_n^k+\dots$
involves an infinite sum.  In general, an infinite sum makes sense as
an operator on $\K[X]$ providing that, for each $n$, all but a finite
number of the monomials $x^K\p^L$ annihilate $\K[X_n]$.

In practice, this is not an issue; such symmetric operators can be
manipulated in the same classical manner as, for example, symmetric
functions on an infinite number of variables. They restrict to
elements of the Weyl algebra $\weyl(X_n)$ for each $n$, and the
notions of order, filtration and grading carry over properly. By
extension, we also call them Weyl operators.

The \emph{wedge product} (or \emph{formal product}) of two
differential operators is defined on standard monomials of the Weyl
algebra by
$$
x^K\p^L \wedge x^M\p^N := x^{K+L} \p^{M+N}.
$$
In other words, the partial derivatives are made to commute with
all the variables, and $\weyl$ becomes polynomial in both sets of
variables $x_i,\p_i$ under the wedge product. In working out the
composition product of two standard forms in the Weyl algebra, the
wedge product is the term of top filtration. The wedge product is
extended to infinite sums by linearity.

\begin{private}

The isobars (Weyl operators degree $0$) are generated by the elements
$x^i\partial_i$. They form a maximal commutative subalgebra.
$1+qx_i\partial_i$ lies there, and its inverse $\sum
\frac{(-q)^k}{\prod_i=1^k (1+iq)}$ lies in the isobars.

The following isobar series returns the constant term of a polynomial:
$T:=\prod_i (\sum_{n=0}^\infty \frac{1}{n!} x_i^n\partial_i^n)$
(Taylor expansion). It cannot be expressed as a series in powers of
$x_i\partial_i$, which makes it difficult to factor as the
exponential of something.

The scalar product can be twisted by an isobar operator $f_0$:
$\scalar{f}{g}:=(f^*f_0g)_1$. To get again a scalar product, $f_0$
should be auto-dual ($f_0^*=f_0$), as well as definite positive. In
this case, the adjoint of an operator $f$ becomes $(f_0 f
f_0^{-1})^*$.
\end{private}

\subsection{Duality in the Weyl algebra and scalar product on polynomials}

\begin{defn}
  The \emph{dual} $f^*$ of a Weyl operator $f$ is defined by
  sesquilinearity on the standard monomials:
  $$
  (x^K \p^L)^*:= x^L \p^K.
  $$
\end{defn}
For example,
$$
(2x_1x_2^2\p_2^2\p_3^3 + 3i x_2\p_1^4)^* = 2x_2^2x_3^3 \p_1\p_2^2 - 3ix_1^4\p_2.
$$
Note that in the literature, the dual of a polynomial $p$ is often
written $p(\p)$ instead of $p^*$~\cite{Garsia_Haiman.OHGR}.

\begin{rem}
  The dual is an anti-morphism with respect to composition:
  \begin{equation}
    (fg)^*=g^*f^*
  \end{equation}
\end{rem}

Denote by $p(0)$ the constant term of a polynomial $p$. For two
polynomials $p$ and $q$, set
\begin{equation}
  \scalar p q := (\overline p^*.q)(0).
\end{equation}
It is easy to check that this defines a scalar product which makes the
monomials $\{x^K\}$ into an orthogonal basis. In fact, we have
\begin{equation}
  \scalar{x^K}{x^L} =
  \begin{cases}
    0  & \text{if $K\ne L$}\\
    K! & \text{if $K=L$}
  \end{cases},
\end{equation}
where $K!=k_1!k_2!\cdots k_n!$.

\begin{prop}
  Let $f$ be a Weyl operator, seen as a linear endomorphism of
  $\K[X]$. Then, the adjoint of $f$ with respect to the scalar product
  $\scalar{.}{.}$ is its dual Weyl operator $f^*$.
\end{prop}
\begin{proof}
  Note first that for a Weyl operator $f$ and a polynomial $p$, we
  have
  \begin{equation}
    (f^*.p)(0) = ((f.1)^*.p)(0)\,;
  \end{equation}
  indeed, $(f-f.1)$ is a pure differential operator, hence
  $(f-f.1)^*.p$ is a pure polynomial with no constant term. Then,
  \begin{multline}
    \scalar{f.p}{q} =
    ((f.p)^*.q)(0)=
    ((fp.1)^*.q)(0)=
    ((fp)^*.q)(0)= 
    (p^*f^*.q)(0) \\
    =(p^*.(f^*.q))(0)=\scalar{p}{f^*.q}\,,
  \end{multline}
  as desired. 
\end{proof}

\subsection{Hit and harmonic polynomials}
\label{subsection.HitsHarmonics}

In this subsection, we fix a set $S$ of homogeneous Weyl operators,
and denote by $S^+$ the operators in $S$ of positive degree.

\begin{defn}
  A polynomial is called \emph{$S$-hit} if there exists some
  polynomials $p_1,\dots,p_k$ and operators $f_1,\dots,f_k\in S^+$
  such that
  \begin{equation}
    p=f_1.p_1+\dots+f_k.p_k.
  \end{equation}
  A polynomial $p$ is called \emph{$S$-harmonic} if it satisfies the
  differential equations
  \begin{equation}
    f^*.p=0, \qquad \text{for all $f \in S^+$}.
  \end{equation}
  The sub-spaces of $S$-hit and $S$-harmonic polynomials are denoted
  respectively by $\hit$ and $\harm$.
\end{defn}

For example, if $S=\{x_1^2+\dots+x_n^2\}$, the harmonics in $\K[X_n]$
are the polynomials $p$ which satisfy
$$
\p_1^2 p + \dots + \p_n^2 p =0;
$$
hence the name \emph{harmonic}~\cite{Garsia_Haiman.OHGR}. If $S$
consists only of polynomials, as in the previous example, then the
space of $S$-hits is simply the ideal generated by $S^+$ in $\K[X_n]$.

\begin{rem}
  Consider an operator $f\in S^+$, any Weyl operator $g$, and a
  polynomial $p$. Then, $fg.p= f. (g.p)$ is hit. Dually, if $p$ is
  harmonic then $gf^*.p=g.(f^*.p)=g.0=0$: it is killed by any operator
  in the left ideal generated by $S^{+*}$ in the Weyl algebra. Hence,
  $S$ and the right ideal $S^+\weyl(X_n)$ generated by $S^+$ in the
  Weyl algebra have the same hit and harmonic polynomials.
  
  Furthermore, if $A^+$ is any subspace of the Weyl algebra such that
  $S^+\subset A^+ \subset S^+\weyl(X_n)$, then $\hit$ is the image
  $A.\K[X_n]$ of $A^+\tensor \K[X_n]$ by the bilinear evaluation map
  $(f,p)\mapsto f.p$.
\end{rem}

\begin{prop}
  A polynomial $p$ is $S$-harmonic if, and only if, it is orthogonal
  to any $S$-hit polynomial $q$. That is, the space of $S$-harmonic
  polynomials is the orthogonal complement of the space of $S$-hit
  polynomials.
\end{prop}
\begin{proof}
  This is a simple variation on the fact that the image of a linear
  morphism is the orthogonal complement of the kernel of its adjoint.
  Let $p$ be a polynomial. By duality, for any polynomial $q$ and
  operator $f\in S^+$, we have
  $$
  \scalar{f^*.p}{q} = \scalar{p}{f.q}.
  $$
  Assume that $p$ is $S$-harmonic; then it is obviously orthogonal
  to any $f.q$ and by linearity to any $S$-hit. Reciprocally, assume
  that $p$ is orthogonal to any hit polynomial, and consider an
  operator $f\in S^+$; then $f^*.p=0$ because it is orthogonal to all
  polynomials.
\end{proof}

\begin{prop}
  \label{prop.harm.hilbert}
  Let $A$ be the graded algebra generated by $S^+$, acting on
  $\K[X_n]$. Then, $\harm$ generates $\K[X_n]$ as an $A$-module. In
  particular,
  \begin{equation}
    \hilb(A).\hilb(\harm) \geq \hilb(\K[X_n]),
  \end{equation}
  with equality if, and only if, $\K[X_n]$ is a free $A$-module.
\end{prop}
\begin{proof}
  It is sufficient to verify, by induction on the degree, that any
  homogeneous polynomial can be written in the form $p=1.p_0 +
  f_1.p_1+\dots+f_k.p_k$ where the $p_i$'s are harmonic. Then,
  $\K[X_n]=A.\harm$ is a graded quotient of $A\tensor\harm$, which
  yields the desired inequality on the Hilbert series.
\end{proof}

\begin{prop}
  Assume that the operators in $S$ are symmetric. Then, any Specht
  polynomial $\specht_F$ is $S$-harmonic. Hence, the space of
  $S$-harmonics contains at least one copy of each irreducible
  representation of the symmetric group $\sg_n$.
\end{prop}
\begin{proof}
  Consider an operator $f\in S^+$. Its dual $f^*$ is an
  $\sg_n$-morphism of negative degree. Hence, $f^*.\specht_F$ lies in
  the same isotypic component as $\specht_F$, and
  $\deg(f^*.\specht_F)<\deg(\specht_F)$. By minimality of the degree
  of $\specht_F$, $f^*.\specht_F=0$.
\end{proof}

As a final example, take for $S$ the $n$ symmetric power-sums on $n$
variables $(p_k:=x_1^k+\dots+x_n^k)_{k=1,\dots,n}$.  Then, the space
$\hit$ of hit polynomials is the ideal generated by the symmetric
polynomials. There exists an explicit Gröbner basis for the
lexicographic term order, which shows that the leading terms of $\hit$
are the non-staircase monomials (see e.g.~\cite{Sturmfels.AIT}).

The harmonics are the polynomials $p$ which satisfy the generalized
harmonic equations
\begin{equation}
  p_k^* .p = \p_1^k.p + \dots + \p_n^k.p=0, \qquad \text{for $k=1,\dots,n$}.
\end{equation}
The Vandermonde $\Delta:=\prod(x_j-x_i)$ is the Specht polynomial
$\specht_{(1,\dots,1)}$, and as such is harmonic. Moreover, any
partial derivative $\p_i$ commute with the $p_k^*$, and it follows
that it stabilizes the space $\harm$ of harmonic polynomials.
Actually, any harmonic is of the form $p^*.\Delta$, for some
polynomial $p$, and if a set $B$ of $n!$ polynomials is a basis of
$\K[X_n]/\sym.\K[X_n]$, then $(p^*.\Delta)_{p\in B}$ is a basis of
$\harm$. Finally, $\K[X_n]$ is a free $\sym(X_n)$-module, and the
Hilbert series of $\harm(X_n)$ is given by
\begin{equation}
  \begin{split}
    \hilb(\harm(X_n)) &= \frac{\hilb(\K[X_n])}{\hilb(\sym(X_n))}\\
    &=
    \prod_{i=1}^{n}\frac{1-t^i}{1-t} =
    (1+t) (1+t+t^2) \cdots(1+t+\dots+t^{n-1}).
  \end{split}
\end{equation}

For further details, we refer to~\cite{Garsia_Haiman.OHGR}.

\commentaire{More details ?}

\subsection{Specialization lemmas}

In this subsection, $U$ is a $\C$-vector space with a distinguished
basis $B$. The goal here is to deal with $q$-deformations of finite
dimensional sub-spaces of $\C(q) \tensor U$ and their specializations.

\begin{lem}
  \label{lem.RankOfSpecialization}
  Let $(u_1(q),\dots,u_k(q))$ be a family of vectors of $\C(q) \tensor
  U$. For all but a finite set of values $q_0\in \C$, namely the roots
  of the denominators of the coefficients of the $u_i$'s on the basis
  $B$, the specialization $q=q_0$ is well defined. In this case,
  \begin{equation}
    \label{eq.RankOfSpecialization}
    \rank_\C     (u_1(q_0),\dots,u_k(q_0))
       \leq
    \rank_{\C(q)} (u_1(q),\dots,u_k(q)) \,.
  \end{equation}
  Moreover, there is only a finite number of values $q_0\in C$ such
  that the inequality above is strict, namely the roots of the
  non-trivial principal minors.
\end{lem}

Let $V(q)$ be a subspace of $\C(q) \tensor U$ of finite dimension. For
any $q_0\in \C$, let $V(q_0)$ be its specialization, that is the
vector space of all meaningful specializations of elements of $V(q)$.
Note that this definition is highly dependent on the basis $B$.

\begin{lem}
  \label{lem.Specialization}
  There exists a basis $B'(q)=(v_1(q),\dots,v_k(q))$ of $V(q)$ such that the
  coefficients of the elements of $B'(q)$ in the basis $B$ are polynomials in
  $q$, and for any specialization $q=q_0$, the family $B'(q_0)$ is a basis
  of $V(q_0)$ (in particular $\dim_\C V(q_0)=\dim_{\C(q)} V(q)$).
\end{lem}
This is a classical result, but let us give a short effective proof. 
\begin{proof}
  First remark that, by the preceding lemma, the vector space $V(q_0)$
  of all meaningful specializations of elements of $V(q)$ is of
  $\C$-dimension at most the $\C(q)$-dimension of $V(q)$. Consequently
  we just have to construct a basis of $V(q)$ without denominators
  such that all specializations are linearly independent. Let us
  construct such a basis.
  
  Choose a $\C(q)$-basis $A(q)$ of $V(q)$ and expand the vectors of
  $A(q)$ over $B$. By a suitable multiplication, we can remove the
  denominators and then suppose that the coefficients of the expansion
  are in $\C[q]$. Now we apply the fraction-free Gram-Schmidt
  orthogonalization to $A(q)$, and get a basis $A'(q) = (a_1(q),
  \dots, a_n(q))$. By dividing each $a_i(q)$ by the greatest common
  divisor of its coefficients, (it is often called the \emph{content}
  of $a_i(q)$) we get a new basis $B'(q)=(b_i(q))$ of $V(q)$.
  
  This basis is orthogonal on $\C[q]$ and thus remains orthogonal for
  any specialization at $q=q_0$. Furthermore, an element $b_i(q)$ of
  $B'(q)$ is content-free, so $b_i(q_0)$ never vanishes. We conclude
  that, for any $q_0$, the family $B'(q_0)$ is a $\C$-basis for
  $V(q_0)$.
\end{proof}

Moreover, since the character table of $\sg_n$ is with integer
coefficients, the action of central idempotents commutes with
specializations. Thus if $V(q)$ is an $\sg_n$-module then
\begin{itemize}
\item $V(q)$ is isomorphic as $\sg_n$-modules to $\C(q) \tensor V(q_0)$ for
  all $q_0$;
\item $\C(q_0)$ and $\C(q_1)$ are isomorphic as $\sg_n$-modules for all $q_0,
  q_1$.
\end{itemize}

\section{The $q$-Steenrod algebra}

The similarities between $\sym$ and $\steen$ leads us to define a new
algebra to interpolate between the two.

\begin{defn}
  \label{defn.qsteen}
  Let $\P k$ be the following $q$-analogue of the $\P[]k$ operator:
  \begin{equation*}
    \P k:=\sum_i x_i^k (1+qx_i\p_i).
  \end{equation*}
  The \emph{$q$-Steenrod algebra} $\qsteen$ is the algebra over
  $\C(q)$ generated by the $\P k, k>0$.
\end{defn}

Note that the specializations at $q=0$ and $q=1$ yield respectively
the symmetric functions and the rational Steenrod algebra (conjugated
by the Thom map):
\begin{itemize}
\item $\P[0]k$ is the usual operator of multiplication by the
  symmetric power-sum $p_k$.
\item $\P[1]k=\P[]k$.
\end{itemize}

\begin{rem}
  It would be tempting to actually introduce two parameters $q$ and
  $q'$:
  \begin{equation}
    \P[q,q']k:=\sum_i x_i^k (q'+qx_i\p_i).
  \end{equation}
  Now, the rational Steenrod algebra itself is also a special case of
  the $(q,q')$-Steenrod algebra, with $q'=0$ and $q=1$: working in
  projective space versus $q$ removes the singularity at $q=\infty$.
  The price to pay is that $\C[q,q']$ is not anymore a principal ideal
  domain; this renders computations, and especially linear algebra,
  far less practical. Since we are mainly interested in the cases
  $q=0$ and $q=1$, we stick in the following to a single parameter.
\end{rem}

\subsection{Structure and dimension of the $q$-Steenrod algebra}

The $q$-Steenrod squares satisfy the following commutation rule:
\begin{prop} \label{eq.commut}
  For $k, l\in \N$, and $q$ complex or formal,  
  \begin{equation}
    [\P k,\P l]=q(l-k)\P{k+l}\,.
  \end{equation}
\end{prop}
For a composition $\mu:=(\mu_1,\dots,\mu_k)$, define $\P\mu :=
\P{\mu_1}\cdots \P{\mu_k}$.  Using the commutation rules, it is
obvious that any $\P\mu$ can be rewritten as a linear combination of
$\P\lambda$'s indexed by partitions. For example,
$$
\P{(1,2)} = q \P{(3)} + \P{(2,1)}\,.
$$
\begin{thm}\label{thm.structSteen}
  Let $q$ be formal or a complex number.
  \begin{itemize}
  \item[(a)] The operators $\P\lambda$, where $\lambda$ runs
    through all integer partitions form a vector space basis of
    $\qsteen$;
  \item[(b)] The Hilbert Series of the $q$-Steenrod algebra is given by
    \begin{equation}
      \label{eq.HilbSteen}
      \hilb(\qsteen) = \hilb(\sym) = \prod_{i>0} \frac{1}{1-t^i}\,;
    \end{equation}
  \item[(c)] The ideal of relations between the $\P k$'s is generated
    by the commutation relations
  \begin{equation}
    [\P k,\P l] - q(l-k)\P{k+l}\,.
  \end{equation}    
  \item[(d)] $\qsteen$ is the enveloping algebra of the Lie algebra spanned
  by the $\P k$.
  \end{itemize}
\end{thm}
This is a straightforward generalization of results
in~\cite{Wood.DOSA.1997} about the rational Steenrod algebra, and
follows from an easy triangularity property with respect to monomial
functions. Namely, consider the expansion of
\begin{equation*}
\begin{split}
  \P{(2,1)} &=  \os(x_1^2 x_2) + (q + 1)\os(x_1^3)\\
  &+ q^2 \os(x_1^5\p_1^2) + q^2 \os(x_1^3x_2^2\p_1\p_2) + q
  \os(x_1^3x_2\p_1) + 2 q (q + 1) \os(x_1^4\p_1) + q \os(x_1^2x_2^2\p_1).
\end{split}
\end{equation*}
Its polynomial part is a symmetric function, and can be expressed in
terms of monomial symmetric function:
$\P{(2,1)}.1=m_{(2,1)}+(q+1)m_{(3)}$. Note that $m_{(2,1)}$ has
coefficient $1$, while $m_{3}$ is indexed by a shorter partition. This
generalizes immediately:
\begin{lem}
  \label{lemma.triangularity}
  The polynomial part $\P\lambda.1$ of $\P\lambda$ is a symmetric
  function of the form
  $$
  \P\lambda .1 = m_\lambda + \sum_\mu c_\mu m_\mu,
  $$
  where $\mu$ runs through partitions of length
  $\length(\mu)<\length(\lambda)$.
\end{lem}

It follows that, except for $q=0$, all the $q$-Steenrod algebras on an
infinite alphabet are isomorphic independently of $q$; only their
action on polynomials differ:
\begin{thm}\label{thm.isomSteen}
  Let $q$ be formal (respectively a non zero complex number). Then,
  the $q$-Steenrod algebra $\qsteen$ is isomorphic as a graded algebra
  to $\C(q)\tensor\steen$ (respectively $\steen$).
\end{thm}
\begin{proof}
  By renormalizing the $q$-Steenrod squares $\tP k:=\frac{1}{q}\P k$,
  we obtain the commutation relations
  \begin{equation}
    [\tP k,\tP l]=(l-k)\tP{k+l}.
  \end{equation}
  Hence the algebra $\qsteen$ has the same presentation with generators and
  relations as $\C(q)\tensor\steen$.
\end{proof}

We consider now the $q$-Steenrod algebra over a finite alphabet.
\begin{prop}
  For any finite alphabet $X_n$,
  \begin{equation*}
    \prod_{i=1}^n \frac{1}{1-t^i} =
    \hilb(\sym(X_n)) \leq
    \hilb(\qsteen(X_n)) \leq
    \frac{1}{(1-t)^n}\prod_{i=1}^n \frac{1}{1-t^i} 
  \end{equation*}
\end{prop}
\begin{proof}
  The first inequality follows again from
  lemma~\ref{lemma.triangularity}. Note that, for $k\geq 1$, any Weyl
  monomial $x^K\p^L$ appearing in $\P k$ satisfies $k_i\geq 2l_i$ for
  all $i$. This property is preserved by multiplication, and it
  follows that $\P\lambda$ can be written as a linear combination of
  orbit-sums $\os(x^{2K+L}\p^L)$, where $\sum k_i+l_i=n$, and the
  $k_i$ are decreasing. The second inequality follows.
\end{proof}
It appears from computations that the first inequality is strict,
starting from degree $n+1$. Namely, for $n\leq 5$, we noticed that the
polynomials $(\P\lambda(X_n).x_1)_{\lambda\partof n+1}$ are always
linearly independent (see Sections~\ref{tables.cones}
and~\ref{subsection.Elementary}), whereas there is a linear relation
between the symmetric functions $(p_\lambda(X_n))_{\lambda\partof
  n+1}$.

The second inequality is very coarse, and could certainly be refined.
For example, one can easily check that
\begin{equation}
  \hilb(\qsteen(X_1))=\frac{1}{1-t}\ll\frac{1}{(1-t)^2}.
\end{equation}
Still, this inequality implies that
$\hilb(\qsteen(X_n))<\hilb(\qsteen(X))$; indeed the dimension of
$\qsteen_d$ is counted by the number $p(d)$ of partitions whose
asymptotic~\cite{Hardy_Ramajuan.1918} is:
\begin{equation}
  p(d)\approx\frac{1}{4 \sqrt3\, n} \exp{\pi \sqrt{\frac{2n}3}}.
\end{equation}


\begin{private}
  The dual of the Steenrod algebra is the ring of symmetric functions
  with the standard product and a twisted coproduct.
  
  Question: what is the dual of the classical symmetric functions in
  Steenrod ?
  
  On peut trouver des élémentaires qui vérifient ce qu'il faut vis à
  vis de la décomp d'alphabets.
\end{private} 

\subsection{Action on polynomials; hit and harmonic polynomials}

We consider now the natural action of $\qsteen$ on $\K[X_n]$.
Following subsection~\ref{subsection.HitsHarmonics} with
$S=\qsteen^+$, we define the spaces $\qhit$ of \emph{$q$-hit
  polynomials} and $\qharm$ of \emph{$q$-harmonic} polynomials. For
$q\ne 0$, the $q$-Steenrod algebra is generated by $\P1$ and $\P2$,
and it follows that the $q$-harmonics can be described by just two
differential equations:
\begin{equation}
  \D1.p=0 \qquad \text{and} \qquad \D2.p=0\,.
\end{equation}

\begin{private}
  The $q$-Steenrod algebra acts on polynomials, $\K[X]$ as well as on
  symmetric functions, $\sym(Y)$ and combinations thereof
  $\K[X]\tensor\sym(Y)$. Here $Y$ can be finite or infinite.
  Consequently, it also makes sense to let the $q$-Steenrod algebra
  act on functions that are symmetric an all but a finite number of
  variables. This could be somewhat generalized for $X$ infinite.
  Indeed, if $p$ is a polynomial in $\K[X]$, its support $X'\subset X$
  is finite, and the image of $p$ by a Steenrod operator lives in
  $\K[X']\tensor\sym(X\ X')$.
\end{private}

Furthermore, it follows immediately from the triangularity property
described in lemma~\ref{lemma.triangularity} that any symmetric
function is $q$-hit.
\begin{prop}
  Let $q$ be formal or a complex number, and $X$ be a commutative
  alphabet. Then, $\qsteen(X).1=\sym(X)$.
\end{prop}

\section{Specializations}
\label{sec.Specializations}

The action of the $q$-Steenrod algebra behaves properly with respect
to specialization. That is, if a polynomial $p(q)$ specializes
properly to $p(q_0)$, then $\P k.p(q)$ and $\D k. p(q)$ specializes
properly to $\P[q_0] k.p(q_0)$ and $\D[q_0]k. p(q_0)$ respectively.
Hence, we can apply the specialization lemmas to obtain some
properties of the $q$-harmonics and $q$-hits with respect to
specialization. In particular we obtain the implication
$(b)\Rightarrow (a)$ of proposition~\ref{prop.goingup}:
\begin{prop}
  Let $q_0\in\C$. Then,
  \begin{itemize}
  \item[(a)] Let $S(q)$ be a set of homogeneous polynomials in
    $\C(q)[X_n]$, which specializes properly at $q=q_0$.  Then, the
    $\qsteen$-module (resp.  $\qsteen^+$-module) $M(q)$ generated by
    $S$ in $\C(q)[X_n]$ specializes to the $\qsteen[q_0]$-module
    (resp.  $\qsteen[q_0]^+$-module) $M(q_0)$ generated by $S(q_0)$.
  \item[(b)] The Hilbert series of $M(q)$ dominates the Hilbert series
    of $M(q_0)$;
  \item[(c)] Let $p\in \C[X_n]$. The Hilbert series of the
    $\qsteen$-module generated by $p$ dominates the Hilbert series of
    the corresponding $\sym(X_n)$-module;
  \item[(d)] Let $B$ be a $\sym(X_n)$-module basis of $\C[X_n]$, e.g.
    the Schubert polynomials or the staircase monomials. Then $B$
    spans $\C(q)[X_n]$ as a $\qsteen$-module; note that this module is
    not necessarily free.
  \item[(e)] $\qhit$ specializes at $q=q_0$ to $\qhit[q_0]$; in
    particular the Hilbert series of $\qhit$ dominates the Hilbert
    series of $\qhit[q_0]$; furthermore, for all but a countable
    number of $q_0\in\C$, there is equality;
  \item[(f)] $\qhit + \harm = \C(q)[X_n]$.
  \end{itemize}
\end{prop}
\begin{proof}
  (a) Fix a degree $d$. Extract from $S(q)$ a finite subset $A(q)$
  which generates linearly all the other elements of $S(q)$ of degree
  less than $d$, and specializes properly at $q=q_0$. Consider the set
  $B(q)$ of all $\P\lambda.p(q)$, where $p$ belongs to $A(q)$ and
  $\lambda$ is a partition of $d-\deg(p)$. Then, $B(q)$ generates
  $M(q)_d$, and specializes to $B(q_0)$ which generates $M(q_0)_d$.
  Hence, $M(q)_d$ specializes to $M(q_0)_d$.

  The others are direct applications of (a).
\end{proof}

Similarly, the following proposition provides the implication
$(b)\Rightarrow (a)$ of proposition~\ref{prop.goingdown}, and more:
\begin{prop}
  \label{prop.harm.specialization}
  \begin{itemize}
  \item[(a)] If $p(q)$ is $q$-harmonic and specializes properly at
    $q=q_0$, then $p(q_0)$ is harmonic;
  \item[(b)] There exists a basis of $\qharm$ which specializes
    properly for any $q_0$;
  \item[(c)] $\qharm$ specializes to a subspace $\qharm(q_0)$ of
    $\qharm[q_0]$
  \item[(d)] The Hilbert series of $\qharm$ is dominated by the
    Hilbert series of $\harm$; furthermore, for all but a countable
    number of $q_0\in\C$, there is equality;
  \item[(d)] $\hit$ and $\qharm$ are in direct sum.
  \item[(e)] The orthogonal projector $\C(q)[X_n] \mapsto
    \harm\tensor\C(q)$ is injective on $\qharm$.
  \end{itemize}
\end{prop}
\begin{proof}
  (a) For any $k$, $\D k.p(q)$ specializes to $\D[q_0]k.p(q)$ for all
  $k$; hence the latter is zero whenever the former is.
  
  (b) Direct application of the specialization
  lemma~\ref{lem.Specialization}.

  (c) Consequence of (a) and (b).
  
  (d) Take $f\in \hit \cap \qharm$. Then, for all $k$,
  $\sum_i(1+qx_i\partial_i) \partial_i^k f=0$.
  
  Write $f=q^{\operatorname{val}(f)}(f_0+qf_1)$, with $0\ne f_0\in\C[X_n]$
  and $f_1\in\C(q)[X_n]$ such that $f_1$ has no $\frac 1 q$
  coefficient. Then, $f_0+qf_1$ is also $q$-harmonic, and it follows
  that $f_0$ is harmonic. This is in contradiction with $f_0$ being
  hit.
\end{proof}

\commentaire{Add comments on what is missing for the conjecture.}

\section{The low degree and truncated cases}

One of the main difficulties is that the $q$-Steenrod algebra on a
finite alphabet $X_n$ is much bigger than $\sym(X_n)$. In this
section, we get around this difficulty by truncating the $q$-Steenrod
algebra to make it share the dimension of $\sym(X_n)$. No truncation
occurs in low degree, so all the results of this section also apply to
the full $q$-Steenrod algebra in degree $d\leq n$.

The \emph{truncated $q$-Steenrod algebra} $\tqsteen(X_n)$ is the
subspace of $\qsteen(X_n)$ spanned by
$(\P\lambda)_{\length(\lambda)\leq n}$. Note that $\qsteen(X_n)$ and
$\tqsteen(X_n)$ coincide in degree $d\leq n$; however $\tqsteen(X_n)$
is not an algebra.


Here also, applying lemma~\ref{lemma.triangularity}, yields
$\tqsteen(X_n).1=\sym$, and it follows that
$\hilb(\tqsteen(X_n))=\hilb(\sym(X_n))$, for $q$ formal or a
complex number.

\begin{problem}
  For which polynomials $p$ do $\tqsteen.p$ and $\qsteen.p$
  coincide ?
  
  For such a polynomial $p$, find a straightening algorithm for
  rewriting an element of $\qsteen.p$ as element of
  $\tqsteen$.
\end{problem}

The proper analogue of $\qhit(X_n)$ is the space of \emph{truncated
  $q$-hit polynomials} $\tqhit(X_n)$ which is spanned by the $f.p$,
where $f\in\tqsteen^+(X_n)$, and $p\in\harm(X_n)$. The polynomials $h$
in the orthogonal $\tqharm(X_n):=\tqhit(X_n)^\bot$ are called
\emph{truncated $q$-harmonics}.  Note that they do not necessarily
satisfy $\D1.h=0$ and $\D2.h=0$.

Reg Wood's conjecture holds for the truncated $q$-Steenrod algebra,
and for $q$ formal:
\begin{thm}
  Let $q$ be formal. Then,
  \begin{itemize}
  \item[(a)] $\hilb(\tqharm(X_n))=\hilb(\harm(X_n));$
  \item[(b)] $\tqharm(X_n)$ is a graded regular representation of the
    symmetric group;
  \item[(c)] $\K[X_n]=\tqharm(X_n)\bigoplus \hit(X_n) =
    \harm(X_n)\bigoplus \tqhit(X_n).$
  \end{itemize}
  Furthermore, similar statements hold in the isotypic and Garnir
  components of $\K[X_n]$.
\end{thm}

\begin{proof}
  The specialization $q=0$ of $\tqsteen$ is $\sym$. Then $\tqhit$
  specializes to $\hit$. Hence, by
  proposition~\ref{prop.harm.specialization}, $\hilb(\tqharm)\leq
  \hilb(\harm)$.

  Furthermore, $\K[X_n]$ is a free $\sym(X_n)$ module, so applying
  proposition~\ref{prop.harm.hilbert} yields
  \begin{equation}
    \begin{split}
      \hilb(\C[X_n]) &\leq \hilb(\tqsteen(X_n)) \hilb(\tqharm(X_n)) \\
      &\leq \hilb(\sym(X_n)) \hilb(\harm) = \hilb(\C[X_n]).
    \end{split}
  \end{equation}
  It follows that equality must hold, and that
  $\hilb(\tqharm)=\hilb(\harm)$. The same reasoning can be
  carried over in any isotypic or Garnir component, and all the
  statements of the theorem follow.
\end{proof}

\begin{cor}
  \label{cor.truncated}
  Let $\hilb[t,\leq d](A)$ denote the truncation of the Hilbert series
  of $A$ up to degree $d$.
  \begin{itemize}
  \item If $q$ is formal, then
      $\hilb[t,\leq n](\qsteen(X_n)) = \hilb[t,\leq n](\sym(X_n))$.
  \item If $q$ is a complex number, then
       $\hilb[t,\leq n](\qsteen(X_n)) \leq \hilb[t,\leq n](\sym(X_n))$.
  \end{itemize}
\end{cor} 
\begin{proof}
  This is a reformulation of the preceding theorem by the fact that
  for $d\leq n$ the algebra $\tqsteen$ and $\qsteen$ coincide.
\end{proof}

\begin{private}
\begin{prop}
  Let $q_0$ be a complex number. Then, the following are equivalent:
  \begin{itemize}
  \item $\dim \qharm[q_0,d]=\dim \harm_d$, i.e.  $\qharm[q_0]$ is the
    specialization at $q=q_0$ of $\qharm$.
  \item $\dim \qhit[q_0,d]=\dim \hit_d$;
  \item Any $q_0$-harmonic of degree $d$ on at most $n$ variables is
    the specialization at $q=q_0$ of a $q$-harmonic on $n$ variables;
  \end{itemize}
\end{prop}
\begin{proof}
  \commentaire{TODO}
\end{proof}

\begin{conjecture}
  All of the above are true for $n$ big enough with respect to $d$.
\end{conjecture}
There seems to be counter examples for $d=2$ and $n\leq 2$, for $d=3$
and $n\leq 4$, ...  \commentaire{Double check this!}
\end{private}

\section{Changing the number of variables}

\begin{private}
  Most propositions of this section could be expressed in the following more
  general setting. Let $A$ be a graded connected (lie-)algebra with a graded
  action on the polynomials on one variable $\K[x]$.  Then, this action
  extends naturally on a symmetric action on $\K[X_n]\isom \K[x]^{\otimes n}$,
  as well as on $\K[X]$ (if $A$ acts on a vector space $V$, and $p$ is a
  primitive element of $A$, then $p$ acts on $V\otimes V$ by $p\otimes 1 +
  1\otimes p$).

  \begin{problem}
    Let $A$ be a graded connected algebra, generated by elements $\P
    k$ of degree $k$ such that
    $$
    [\P k,\P l] = \lambda_{i,j} \P {k+l}
    $$
    for all $k,l$. When are the harmonics for the action of $A$ on
    $\K[X_n]$ a regular representation of the symmetric group $\sg_n$?
  \end{problem}

  \begin{rem}
    Since the descending operators are alphabets derivation, the
    product of two harmonics $h(X)$ and $f(Y)$ is harmonic whenever
    $X$ and $Y$ are disjoint.
  \end{rem}

  \begin{rem}
    Let $p$ be a polynomial of degree $d$, and whose support is of
    size $n$.
    
    If $p$ is harmonic on $k$ variables, it is harmonic for any higher
    number of variables.
    
    If $p$ is hit on $k$ variables, it is hit on any lower number of
    variables.  If moreover $k\geq n+d$ then $p$ is hit for any
    (possibly infinite) number of variables.
  \end{rem}
  Can we improve on this bound $n+d$ ?
\end{private}

In the preceding section, we proved that conjecture~\ref{conj.main}
holds when the number of the variable is greater than the degree. The
goal here is to relate such $q$-harmonics on a large number of
variable with $q$-harmonics on fewer variables, trying to take down as
much information as possible. In particular we prove the equivalence
$(a)\iff(c1)$ of proposition~\ref{prop.goingup}.

\begin{prop}
  Let $X$ be a set of variables and $y\notin X$ be another variable.
  Let us denote by $\qharm(X,y)_{\deg(y)\leq d}$ the subspace of
  $\qharm(X,y)$ of polynomials of degree at most $d$ in the variable
  $y$\commentaire{which are harmonic for the alphabet $X\cup\{y\}$}.
  The isomorphism
  \begin{equation}
    \pi_{y^d} \ :
    \left\{
      \vcenter{\xymatrix@C=0mm@R=0mm{
        *{\K(X,y)_{\deg(y)\leq d}/\K(X,y)_{\deg(y)\leq d-1}}  
             & \ar@{^{(}->>}[rr] & *{\hskip1cm} & & *{\K(X)}\\
        *{f} & \ar@{|->}[rr]     &              & & *{\coeff(f,y^d)}
        }}
    \right.
  \end{equation}
  \commentaire{hook two head arrow ?}  
  defines an injective mapping :  
  \begin{equation}
    \label{eq.InjHarm}
    \tilde\pi_{y^d} \ :\
       \qharm(X,y)_{\deg(y)\leq d}/\qharm(X,y)_{\deg(y)\leq d-1}
       \hookrightarrow \qharm(X)\,.
  \end{equation}
\end{prop}
\begin{proof}
  The injectivity of $\tilde\pi_{y^d}$ follows from the definition.
  The only thing to prove is that the image $\tilde\pi_{y^d}$ of a
  $q$-harmonic polynomial of degree $d$ in $\K[X,y]$ is in
  $\qharm(X)$. It comes from the fact that the $\P k^*$ are primitive,
  that is
  \begin{equation}
    \D k(X,y) = \D k(X) + \D k(y)\,.
  \end{equation}
  \begin{equation}
    \left(
      \sum_{x\in X} (1+qx\partial_x)\partial_x^k
      + (1+qy\partial_y)\partial_y^k
    \right). f = 0\,.
  \end{equation}
  Write $f = f_0y^d + f_1$ with $\deg_y(f_1)<d$. The identification of
  the homogeneous components of degree $d$ in $y$ in the preceding
  equation proves that
  \begin{equation}
    \left(
      \sum_{x\in X} (1+qx\partial_x)\partial_x^k
    \right). f_0 = 0\,.
  \end{equation}
  This ends the proof. 
\end{proof}

\bigskip

The following proposition yields as corollary the equivalence
$(a)\Rightarrow (c1)$ and by an easy induction the implication
$(a)\Rightarrow (c2)$ of proposition~\ref{prop.goingdown}.
\begin{prop}
  \label{prop.equivShortString}
  Let $q$ be generic. Then, the following are equivalent
  \begin{itemize}
  \item[(a)] Conjecture~\ref{conj.main} holds;
  \item[(b)] For all $n$, any $q$-harmonic on $n$ variables has degree
    at most $n-1$ on each variable;
  \item[(c)] For all $n$, the mapping
    \begin{equation}
      \tilde\pi_{y^d} \ :\
      \qharm(X_n,y)_{\deg(y)\leq d}/
      \qharm(X_n,y)_{\deg(y)\leq d-1}
      \hookrightarrow \qharm(X_n)\,,
    \end{equation}
    is an isomorphism for $d\leq n$. For $d > n$,
    $\qharm(X_n,y)_{\deg(y) = d} = \{0\}$.
  \end{itemize}
%
\end{prop}

\begin{proof}
  $(c)\Rightarrow (b)$ is obvious.
  \medskip
  
  $(a)\Rightarrow (c)$: For any degree $d$, the injectivity of
  $\tilde\pi_{y^d}$ implies that
  \begin{multline}
      \label{eq.harmInjective}
      \hilb(\qharm(X_{n},y)_{\deg(y)\leq
        d}/\qharm(X_{n},y)_{\deg(y)\leq d-1})\\
      \leq t^d \hilb(\qharm(X_{n})),
  \end{multline}
  with equality if, and only if, $\tilde\pi_{y^d}$ is an isomorphism.
  On the other hand, we assumed that conjecture~\ref{conj.main} holds.
  Identifying $y$ with $x_{n+1}$, it follows that
  \begin{equation}
    \label{eq.harmSeries}
    \hilb(\qharm(X_{n+1})) = (1+t+\dots+t^{n})\hilb(\qharm(X_{n})).
  \end{equation}
  If for some $d\leq n-1$, the inequality in
  equation~\ref{eq.harmInjective} above is strict, the terms of degree
  $d$ in equation~\ref{eq.harmSeries} will differ. Then, by a similar
  reasoning, the existence of a non-trivial component
  $\qharm(X_{n},y)_{\deg(y)\leq
    d}/\qharm(X_{n},y)_{\deg(y)\leq d-1})$ for $d\geq n$ would
  contradict equation~\ref{eq.harmSeries}.
  \medskip
  
  $(b)\Rightarrow (a)$: By the assumption, and the injectivity of
  $\tilde\pi_{x_n^d}$, for any $n$,
  \begin{equation}
    \label{eq.shortstrings}
    \hilb(\qharm(X_{n+1})) \leq (1+t+\dots+t^{n})\hilb(\qharm(X_{n})).
  \end{equation}
  By induction 
  \begin{equation}
    \hilb(\qharm(X_{n+1})) \leq (1+t) (1+t+t^2)
    \cdots(1+t+\dots+t^{n})
    = \hilb(\harm).
  \end{equation}
  Suppose now that, for a given $n_0$, this inequality is strict:
  \begin{equation}
    \hilb(\qharm(X_{n_0+1})) < (1+t) (1+t+t^2)\cdots (1+t+\dots+t^{n_0}).
  \end{equation}
  Let $d$ be the valuation of the difference. Applying again
  induction, we obtain that
  \begin{equation}
    \hilb(\qharm(X_{d+1})) < (1+t) (1+t+t^2)\cdots (1+t+\dots+t^{d}),
  \end{equation}
  with the inequality still appearing in degree $d$. However, the
  number of variables now exceeds $d$, and this is in contradiction
  with Corollary~\ref{cor.truncated}.
\end{proof}


\subsection{Strings}

We introduce a little bit more formalism to analyze when $q$-harmonic
polynomials in $\K[X_{n+1}]$ have variables of degree at most $n$.

\begin{private}
  Truncated -> Non truncated: show that there is no string of length
  $\geq n$; provide a straightening algorithm to rewrite $P.p$ for
  $P$.
\end{private}

\begin{prop}
  \label{prop.String}
  Consider $f\in\K[X,y]$, and write its expansion on the variable $y$
  as
  \begin{equation}
    f(X,y) = \sum_{i\geq 0} \frac1{i!}f_i(X) \, y^i\,.
  \end{equation}
  Then, $f$ is $q$-harmonic if, and only if, for any $i\geq 0$ and
  $k\geq1$,
  \begin{equation}
    \D k.f_i = - (1+qi)\, f_{i+k} \,.
  \end{equation}
\end{prop}

\begin{proof}
  Since $f$ is $q$-harmonic, $\D k.f=0$. Hence,
  \begin{equation}
    \sum_{i\geq 0} \frac1{i!} \, (\D k.f_i)  y^i +
    \sum_{i\geq 0} \frac1{i!} \, f_i \,(\D k.y^i)
    =0\,,
  \end{equation}
  and then
  \begin{equation}
    \sum_{i\geq 0} \frac1{i!}     \, (\D k.f_i)        \, y^i +
    \sum_{i\geq k} \frac1{(i-k)!} \, f_{i} \, (1+q(i-k)) \, y^{i-k}
    =0\,.
  \end{equation}
  Setting $j:=i-k$ in the second sum, it follows that
  \begin{equation}
    \sum_{i\geq 0} \frac1{i!} \, (\D k.f_i)      \, y^i +
    \sum_{j\geq 0} \frac1{j!} \, f_{j+k} \, (1+qj) \, y^j
    =0\,.
  \end{equation}
  Equating the coefficients of $y$ yields the desired equality, for
  all $i$.
\end{proof}

A remarkable consequence of proposition~\ref{prop.String} is that a
$q$-harmonic polynomial $f$ can be reconstructed from its tail $f_0$.
This motivates the following definition:
\begin{defn}
  Let $g\in\K(X)$ be a homogeneous polynomial. The \emph{string}
  generated by $g$ is the sequence $F = \strng(g) := (f_i)_{i\geq 0}$
  defined by
  \begin{equation}
    \label{eq.defn.string}
    f_0 := g,
    \qquad\text{ and for all $i>0$, }\quad
    f_i := -\D i .g\,.
  \end{equation}
  The \emph{length} of the string $F=(f_i)$ is defined by
  \begin{equation}
    \length(F) := 1 + \max\{i\ |\ f_i\neq 0\}. 
  \end{equation}
  The \emph{head} of the string is $f_{\length(f)}$.  
\end{defn}
The length of a string is nothing but the degree in the variable $y$
of the associated polynomial $f(X,y) := \sum \frac1{i!}f_i(X) \, y^i$.
Obviously, if $g$ is of degree $d$ then $f_i=0$ for all $i>d$,
consequently $\length(\strng(g))<\deg(g)$ and the length is well
defined.

\begin{defn}
  A string $F=(f_i)_i$ is called \emph{harmonic} if
  \begin{equation}
    \D k . f_i = - (1+qi)\, f_{i+k}, \qquad\text{for all $i,k$}\,.
  \end{equation}
\end{defn}

Thanks to the relations in the $q$-Steenrod algebra, it is sufficient
in the above definition to check that $\D 1$ acts properly:
\begin{prop}
  Let $F=(f_i)_i$ a string. Then $F$ is harmonic if, and only if,
  \begin{equation}
    \D 1 . f_i = - (1+qi)\, f_{i+1}, \qquad\text{for all $i$}\,.
  \end{equation}
\end{prop}
\begin{proof}
  This is a simple consequence of the commutation relation. 
  Namely, we have
  \begin{equation}
    \begin{split}
      \D k . f_l = -\D k\D l.f_0 &= q(l-k)\D{(k+l)}.f_0 -\D l\D k.f_0\\
                                 &= -q(l-k) f_{k+l} + \D l f_k\,.
    \end{split}
  \end{equation}
  The proposition follows by an easy induction. 
\end{proof}
For example,
$$
\D2 . f_1 = -\D2\D1 . f_0 = -q\D3 . f_0 - \D1\D2 . f_0
          = qf_3 -(1+2q) . f_3 = -(1+q) f_3\,.
$$

Here are some basic properties of harmonic strings:
\begin{prop}
  Let $F=(f_i)_i$ be an harmonic string. Then,
  \begin{itemize}
  \item[(a)] The head $f_{\length(F)}$ of the string is harmonic;
  \item[(b)] $f_i\ne 0$ whenever $0\leq i \leq \length(F)$ and $q\ne-1/i$.
  \end{itemize}
\end{prop}
\begin{proof} These are obvious consequences of the definitions: 
  for $k>0$, one has $\D k(f_{\length(F)})=f_{\length(F)+k}=0$ and
  $\D{\length(F)-i}(f_i)=-(1+qi) f_{\length(F)}\ne0$.
\end{proof}

\subsection{Length of the strings}

Proving that no $q$-harmonic polynomial in $\K[X_{n+1}]$ has variable
degree $\leq n$ is obviously equivalent to proving that all strings on
$\K[X_n]$ have length $\leq n-1$. The rest of this section aims at
some partial results in this direction.

\begin{private}
\begin{lemma}
  Let $q$ be generic or $q\in\R^+$. Let $p$ be a polynomial, and
  $(f_i)_i$ be a string such that $f_0\in\qsteen^+.p$. Then,
  $f_{l(f)}$ is orthogonal to $p$.
\end{lemma}
\begin{proof}
\end{proof}
\end{private}

\begin{prop}
  \begin{itemize}
  \item[(a)] Let $q=0$, and $F$ be a $0$-harmonic string on $\K[X_n]$.
    Then, the length of $F$ is at most $n$.
  \item[(b)] Let $q$ be formal, and $F$ be a $q$-harmonic string on
    $\K[X_n]$ of length $d$, so that $f_0$ is of the form $f_d.p$,
    where $p$ is symmetric. Then, the length of $F$ is at most $n$.
  \end{itemize}
\end{prop}
(a) is a classical result, and (b) is a little extension to certain
$q$-harmonic strings using a similar proof. Note that a $q$-harmonic
string $F$ for $q$ generic will specialize at $q_0$ to a string
$F(q_0)$ of same or smaller length. In particular, the above
proposition shows that the string $F(0)$ will be of length at most
$n-1$, so all the $f_i(q), i\geq n$ are divisible by $q$.

\begin{proof}
  (a) Let $F$ be a $0$-harmonic string on $\K[X_n]$ of length $k>n$
  and degree $d=n+1$. On $n$ variables, there is an explicit relation
  between symmetric polynomials:
  \begin{equation}
    \label{eq.powersumsRelations}
    \sum_{|\lambda|=d} (-1)^{\ell(\lambda)} c_\lambda \D[0]\lambda = 0
  \end{equation}
  where the $c_\lambda$ are positive constants
  (cf.~\cite{MacDonald.SF.1995}). This comes from the fact that, on
  $n$ variables, the $n$ first power sums generate $\sym(X_n)$ as a
  $\Q$-algebra. Applying $\D[0]\lambda$ to $f_{k-d}$ yields
  $(-1)^{\ell(\lambda)} d_\lambda f_k$ where $d_\lambda$ are
  polynomials in $k$. Hence a contradiction.

  (b) Let $F$ be a string. Thanks to the string relations,
  \begin{equation}
    \D\lambda(f_{k-|\lambda|}) = (-1)^{\ell(\lambda)} R_\lambda(q) f_k
  \end{equation}
  for some polynomial $R_\lambda(q)$ with positive coefficients and
  constant term $1$.
  \begin{private}
    (it is possible to write an explicit formula $R(q) = q_\lambda$).  
  \end{private}
  
  Assume now that $d$ is the length of $F$, and that $f_0$ is of the
  form $f_d p$ with $p$ symmetric. Write the matrix $(\D\lambda(m_\mu
  f_d))_{|\lambda|=|\mu|=n+1,\ell(\mu)\leq n}$.  This matrix has one
  more column than rows. Furthermore, for $q=0$ it has maximal rank.
  Hence, for $q$ generic it is also of maximal rank, and there exists
  a unique linear combination of the rows, up to a scalar coefficient
  in $\C(q)$. We pick up such a linear combination whose coefficients
  $C_\lambda$ are polynomials in $\C[q]$ which are relatively prime in
  their set:
  \begin{equation}
    \forall \mu\,,\qquad
    \sum_\lambda (-1)^{\ell(\lambda)} C_\lambda(q) \D\lambda(m_\mu f_d) = 0.
  \end{equation}
  At $q=0$, this relation specializes, up to a coefficient $c\ne 0$,
  to the relation between the $p_\lambda$'s described in
  equation~\ref{eq.powersumsRelations}. We cancel out this
  coefficient, so that $C_\lambda(0)=c_\lambda$.

  Applying this relation on $f_{k-d}$ yields
  \begin{equation}
    0 = \sum_\lambda (-1)^{\ell(\lambda)} C_\lambda \D\lambda(f_0) =
    \left(\sum_\lambda C(q)_\lambda R_\lambda(q)\right) f_d.
  \end{equation}
  However,
  \begin{equation}
    \sum_\lambda C_\lambda(0) R_\lambda(0) = \sum_\lambda c_\lambda \ne 0.
  \end{equation}
  Hence, $f_d=0$, which is a contradiction.
\end{proof}

\begin{private}
\begin{itemize}
\item La conjecture est aussi équivalente à l'énoncé : Les harmoniques de
  degrée $d$ en $n=d+1$ variables engendrent tous les harmoniques de
  degrée $d$ sous l'action de $\K[\sg_n]$. 
  

  
  
\item Contre exemple à $q=-1$ :
  
  Le polynôme $x_1^2-x_2^2$ est $-1$-harmonique et il engendre les harmoniques
  en 3 variables. Mais, le polynôme $(x_1-x_2)(x_3-x_4)$ n'est pas engendré.
 
\end{itemize}
\end{private}

\subsection{Complete analysis of the 2 variables case}

We work over the alphabet $X=\{x\}$, and make use of the strings
machinery. For simplicity, let us use the notation $x^{[d]} :=
x^d/d!$. The Steenrod operators act by the rule:
\begin{equation}
  \D k . x^{[d]} = (1+q(d-k))\, x^{[d-1]}\,.
\end{equation}

Let $(f_i)$ be a harmonic string. Then,
\begin{equation}
  -\D1 .f_k = -\D1 . (-\D k . f_0) = -(1+kq) \D{k+1}.f_0\,.
\end{equation}
Define $A_k$ so that
\begin{equation}
  A_k(f_0) = (\D 1 \D k + (1+kq) \D {k+1}). f_0 = 0\,,
\end{equation}
and suppose that $f_0=x^{[d]}$. Then, $A_k(f_0)$ is of the form $a_k
x^{[d-k]}$, with
\begin{equation}
  a_k = (1+q(d-k-1))(1+q(d-k) +1+kq)\,.
\end{equation}
Hence, $F=\strng(x^{[d]})$ is harmonic if, and only if,
\begin{equation}
\text{for all $0< k< d$} \qquad (2+qd)(1+q(d-k-1)) = 0\,.
\end{equation}
So, there are two possible solutions, namely $d<2$ and $q=-2/d$.

Going back from strings to harmonic polynomials, we get the proposition:
\begin{prop}
  Let $q$ be formal or a complex number. Then, the $q$-harmonic
  polynomials in $\K[x,y]$ are spanned by
  \begin{equation}
    1 \quad\text{and}\quad (x-y)\,.
  \end{equation}
  if $q$ is not of the form $q=-2/d$ for $d$ integer, $d\geq 2$, and
  by
  \begin{equation}
    1,\quad (x-y),\quad\text{and}\quad (x-y)(x+y)^{d-1}\,.
  \end{equation}
  otherwise.
\end{prop}
In particular, if $n\geq2$ and $d>n(n-1)/2$ the polynomial
$f:=(x_1-x_2)(x_1+x_2)^{d-1}$ is $q_0$-harmonic for $q_0:=-2/d$.
However, the highest degree of a harmonic polynomial in $n$ variables
is $n(n-1)/2$. This shows that conjecture~\ref{conj.main} would not
hold without condition on $q$.
\begin{prop}
  Take $d>2$, and $n>=2$ such that $n(n-1)/2<d$. Then, the statement
  of conjecture~\ref{conj.main} would not hold for $q_0:=-2/d$.
\end{prop}

\begin{private}

\subsubsection{Cas de 3 variables}

Soit $h$ un harmonique en $3$ variables dont la tête est singulière.
Alors il n'est pas combinaison linéaire d'harmoniques en $2$
variables. En effet, ces harmoniques seraient singuliers pour un $q$
différent.

Si un polynome de degré $d>2$ dont la tête en facteur de $z$ est
harmonique alors sa tête est singulière. Exemple : Le polynome
$$3(x^2-y^2)z + x^3 - y^3$$
$$(x - y) (x y + 3 x z + 3 y z + x^2 + y^2 )$$
est $-1$ harmonique.
Mais, il n'y a pas de polynome harmonique dont la tete est
$(x-y)(x+y)^2 z$. On doit pouvoir traiter la suite (tête =
$(x-y)(x+y)^k z$) par récurence ?

$$P^*_{i,-1}((x-y)(x+y) z^{i+1} )$$
Plus généralement :
$$P^*_{i,-2/(d+1)}((x-y)(x+y)^d z^{i+(d+1)/2} )$$

Proof : Par multiplicativité des alphabets à partir de
$P^*_{i,-2/(d+1)} ((x-y)(x+y)^d) = 0$ et $P^*_{i,q} (z^{i-1/q}) = 0$.

À $q=-1$ on a $(x-y)(x^2 + y^2 + 4xy + 3z^2)$, modulo le Vandermonde
il est égal à $ x^3 + 3x^2z - y^3 - 3y^2z$ que l'on a trouvé plus haut
!!!

À $q=-1/(n-1)$ $Vandermonde*e1$ est harmonique pour $n$ variables.
\end{private}

\section{Why is the $q$ case more difficult than the $q=0$ case?}

The difficulties when trying to prove conjecture~\ref{conj.main}
arises from the fact that the algebra $\qsteen(X_n)$ is much bigger
than $\sym(X_n)$. We cite here some consequences of this; in this
section we will expand more on the last two items.
\begin{itemize}
\item $\K[X_n]$ is not a free $\qsteen(X_n)$ module;
\item The $\qsteen$-modules generated by two polynomials $p_1$ and
  $p_2$ may be non-isomorphic, even when $p_1$ and $p_2$ are of the
  same degree, or in the same isotypic or Garnir component;
\item There are no proper $q$-analogs of elementary symmetric
  functions $e_k$ which vanishes on $X_n$ when $k>n$;
\item Their are very few operators which commute with the action of
  $\qsteen$.
\end{itemize}
With this in view, conjecture~\ref{conj.main} almost comes as a
surprise: why isn't the increased size of $\qsteen(X_n)$
counterbalanced by fewer $q$-harmonics ?
 
\commentaire{More examples and counter examples...}

\subsection{Elementary symmetric functions do not exist}
\label{subsection.Elementary}

As mentioned in section~\ref{subsection.HitsHarmonics}, one way to
treat the classical case $q=0$ is to construct an explicit Gröbner
basis of the ideal $\hit$. This construction relies on the existence
of the elementary $e_k$ and complete $h_k$ symmetric functions which
satisfy for any $k$
\begin{equation}
  \sum_{i=1}^k e_i h_{k-i}=0, \qquad
  e_k(X+Y) = \sum_{i=1}^k e_i(X) e_{k-i} (Y),
\end{equation}
and, last but not least,
\begin{equation}
  e_k(X) = 0, \qquad \text{whenever $k>|X|$}.
\end{equation}
Finding $q$-analogs of the elementary and complete symmetric functions
which satisfy the first two conditions is feasible, for example by
considering $\qsteen$ as an appropriate quotient of the ring of non
commutative symmetric functions. On the other hand, the third
condition does not transfer properly. Indeed, it appears from our
computations that, for $n\leq 5$, the operators
$(\P\lambda(X_n))_{\lambda\partof n+1}$ are linearly independent (see
Sections~\ref{tables.cones}). Hence, no linear combination of them
vanishes on $X_n$, and can serve as a reasonable $q$-analog of
$e_{n+1}$.

\subsection{Schubert Polynomials do not exist}
\label{subsection.Schubert}

\newcommand{\forallini}{\text{for $1\leq i \leq n-1$,}}
\newcommand{\foralliji}{\text{for $|i-j|>1$,}}
\newcommand{\forallind}{\text{for $1\leq i \leq n-2$.}}
In the classical case, the quotient $\K[X]/\sym^+(X)$ is not only
isomorphic to the regular representation of $\sg_n$ but also to many
deformations of it; among them one can find Hecke algebras. The more
general construction has been described by Lascoux and Schützenberger
in \cite{LascouxSchutzenberger87}, with a $5$-parameters deformation.

Recall that the symmetric group is generated by the elementary
transpositions $\sigma_i$ which exchange $i$ and $i+1$ for $i<n$. A
presentation is given by the relations
\begin{alignat}{2}\label{eq.presentSGn}
  \sigma_i^2&=1,                                \qquad&&\forallini\notag\\
  \sigma_i\sigma_j&=\sigma_j\sigma_i,           \qquad&&\foralliji\\
  \sigma_i\sigma_{i+1}\sigma_i &=\sigma_{i+1}\sigma_i\sigma_{i+1},
  \qquad&&\forallind\notag
\end{alignat}%
The last two relations are called the \emph{braid relations}, since
they give a presentation of the braid group. A decomposition $\sigma =
\sigma_{i_1}\sigma_{i_2}\cdots\sigma_{i_k}$ of minimal length is called
a \emph{reduced word} for $\sigma$.

The Hecke algebras $H_n(v_1,v_2)$ are a $2$-parameters family of
quotients of dimension $n!$ of the braid algebra, generated by $T_i$
with the relations
\begin{alignat}{2}\label{eq.presentHecke}
  (T_i-v_1)(T_i-v_2)&=0,                \qquad&&\forallini\notag\\
  T_iT_j&=T_jT_i,                       \qquad&&\foralliji\\
  T_iT_{i+1}T_i &=T_{i+1}T_iT_{i+1},    \qquad&&\forallind\notag
\end{alignat}%
Note that, as for the $(q,q')$-Steenrod algebra, the structure of the
algebra depends only on the quotient $v_1/v_2$.

Here, the useful special case is $v_1=v_2=0$. Indeed, $H_n(0,0)$ acts
on polynomials by the \emph{divided differences operators}
\begin{equation}
  \label{eq.diffdiv}
  d_i(f) := \frac{f-\sigma_i f}{x_i-x_{i+1}}\,. 
\end{equation}
To avoid confusion with the derivations we use $d_i$ instead of the
usual notation $\partial_i$. Note that the $d_i$'s are homogeneous
operators of degree $-1$ on polynomials.

Given a reduced word $\sigma_{i_1}\sigma_{i_2}\cdots\sigma_{i_k}$ for a
permutation $\sigma$, define $d_\sigma:=d_{i_1}d_{i_2}\dots d_{i_k}$.
Thanks to the braid relations, $d_\sigma$ is independent of the choice
of the reduced word for $\sigma$. The family
$(d_\sigma)_{\sigma\in\sg_n}$ happens to be a basis of $H_n(0,0)$.

\medskip

The key point here is the following proposition:
\begin{prop}
  \label{prop.DivDiffSym}
  The divided differences commute with the action of symmetric
  functions: for any $p\in\sym(X_n)$, any $i<n$ and any
  polynomial $f$,
  \begin{equation}
    d_i(p f) = p \,d_i(f)\,.
  \end{equation}
  The quotient $\K[X_n]/\sym^+(X_n)$ is isomorphic to the regular
  representation of the divided differences algebra $H_n(0,0)$.
\end{prop}
This construction was first defined in geometry, where the quotient
$\K[X_n]/\sym^+(X_n)$ is interpreted as the cohomology ring $H^*(F,
\Z)$ of the flag manifold~$F$~\cite{Demazure74,BernsteinGelfandGelfand73}.

Using the divided differences, one can construct a basis of
$\K[X_n]/\sym^+(X_n)$ composed of the so-called Schubert polynomials:
\begin{defn}
  Let $\rho := (n-1,n-2,\dots,1,0)$, so that $X^\rho =
  x_1^{n-1}x_2^{n-2}\dots x_{n-1}$. For each permutation
  $\sigma\in\sg_n$ the \emph{Schubert polynomial} $\schub_\sigma$ is
  defined by
  \begin{equation}
    \label{eq.defSchubertPol}
    \schub_\sigma := d_{\sigma\omega}(x^\rho)\,,
  \end{equation}
  where, for technical reasons, $\omega$ has been chosen to be the
  maximal permutation of $\sg_n$.
\end{defn}
For more details about Schubert polynomials, we refer
to~\cite{MacDonald.1991}. Note that they are denoted by $\sg_\sigma$
in~\cite{MacDonald.1991}, and by $X_\sigma$ in Lascoux and
Schützenberger
work~\cite{LascouxSchutzenberger82,LascouxSchutzenberger87}.

The main theorem is
\begin{thm}[Lascoux-Schützenberger, 1982]
  The family of Schubert polynomials
  $(\schub_\sigma)_{\sigma\in\sg_n}$ is a basis of $\K[X_n]$ as a free
  $\sym$-module.  In particular, the family of Schubert polynomial
  $(\schub_\sigma)_{\sigma\in\sg_n}$ is a basis of $\K[X_n]/\sym^+$.
\end{thm}
Moreover the divided difference machinery provides a very efficient
algorithm to decompose a polynomial in this basis. 

\bigskip

This machinery relies heavily on the commutation property of
proposition~\ref{prop.DivDiffSym}. One strategy to prove the
conjecture~\ref{conj.main} is to search for a family of linear
operators of degree $-1$ which commute with the action of $\qsteen$,
hoping that they will generate a $q$-analog of the divided differences
algebra.  Unfortunately, we obtained by a computation that no such
operator exists for, for example, $n=1,2,3,4$.

Maybe this machinery could still be made to work with some weaker skew
commutation condition. A precise statement of this might be:
\begin{problem}
  Does there exist an operator $T$ of degree $-1$ on polynomials such
  that for any $f\in\qsteen$, there exists $g\in\qsteen$ such that
  \begin{equation}
    f\, T = T\, g\,.
  \end{equation}
  What is the structure of the algebra generated by such operators ?
\end{problem}

For a similar reason, we could not generalize the fact that $\harm$ is
spanned by the derivatives of the Vandermonde. Indeed, there are no
reasonable $q$-analogs of the derivations which commute with the
operators $\D k$'s.

\begin{private}

For $>=$: try to analyze the q-higher Spechts and give an explicit
description of a basis of $\qharm$


One approach is to search for a nice twist sending the harmonics on
the $q$-harmonics. See the file isobar.mu for notes on this.


Let $X$ and $Y$ be two disjoint alphabets, and let $k:=|X|+1$. In
$\sym$, the following relation holds between the $e_i$ and
$h_i$~\cite[p.12]{Sturmfels.AIT,MacDonald.SF.1995}:
\begin{equation}
  \label{eq.GbHit}
  \sum_{i=0}^k (-1)^i e_i(X+Y) h_{k-i}(Y) = 0
\end{equation}
Now, take $X:=(x_1,\dots,x_{k-1})$, $Y:=(x_k,\dots,x_n)$ and let
$m:=x_1^{m_1}\dots x_n^{m_n}$ be a monomial in $\K[X_n]$ such that
$m_k\geq k$. Let $p:=m/x_k^k$. Applying the relation above yields:
\begin{equation}
  h_k(x_k,\dots,x_n) . p
  = - \sum_{i=1}^k (-1)^i e_i(x_1,\dots,x_n) . (h_{k-1}(x_k,\dots,x_n).p)
\end{equation}
Under the lexicographic order with $x_1>\dots,>x_n$, the leading term
of the left hand side is $m$, while the right hand side is in $\hit$.
It follows that the set of standard monomials with respect to $\hit$
is a subset of the staircase monomials. Actually, the relation above
provides precisely a Gröbner basis of the ideal $\hit$ generated by
symmetric functions, the standard monomials being all the staircase
monomials.

\begin{problem}
  Can we find some $q$-analogue in $\qsteen$ of the
  relation~\ref{eq.GbHit} ?
\end{problem}
If yes, we could generalize the above trick. To be precise, the
coefficient for $m$ in $h_k(x_k,\dots,x_n) . p$ would typically be
some polynomial in $q$, say $1+qm_k$, which should not vanish; we
could probably deduce an explicit characterization of the badly
behaving $q_0$. Then, for all the other $q_0\in \C$, we would obtain
$\hilb(\qharm[q_0])\leq\hilb(\harm)$.

Direction: compute with Maple a Gröbner basis in the Weyl algebra for
the right ideal generated by the Steenrod Algebra.
\end{private}

\include{tables}

\bibliographystyle{alpha}
\bibliography{main}

\def\cprime{$'$}
  \ifx\undefined\allcaps\def\allcaps#1{#1}\fi\ifx\undefined\allcaps\def\allcap%
s#1{#1}\fi
\begin{thebibliography}{{The}96}

\bibitem[BGG73]{BernsteinGelfandGelfand73}
I.~N. Bern{\v{s}}te{\u\i}n, I.~M. Gel{\cprime}fand, and S.~I. Gel{\cprime}fand.
\newblock Schubert cells, and the cohomology of the spaces ${G}/{P}$.
\newblock {\em Russian Math. Survey}, 28(3(171)):3--26, 1973.

\bibitem[Dem74]{Demazure74}
Michel Demazure.
\newblock D\'esingularisation des vari\'et\'es de {S}chubert
  g\'en\'eralis\'ees.
\newblock {\em Ann. Sci. \'Ecole Norm. Sup. (4)}, 7:53--88, 1974.
\newblock Collection of articles dedicated to Henri Cartan on the occasion of
  his 70th birthday, I.

\bibitem[FH96]{Fulton_Harris.RT}
William Fulton and Joe Harris.
\newblock {\em Representation theory}, volume 129 of {\em Graduate Texts in
  Mathematics}.
\newblock Springer-Verlag, New York, 1991 - 1996.
\newblock A first course, Readings in Mathematics.

\bibitem[GH94]{Garsia_Haiman.OHGR}
A.~M. Garsia and M.~Hamain.
\newblock Orbit harmonics and graded representations.
\newblock {U}{C}{S}{D} lecture notes, UCSD, 1994.

\bibitem[HR18]{Hardy_Ramajuan.1918}
G.~H. Hardy and S.~Ramajuan.
\newblock Asymptotic formulae in combinatory analysis.
\newblock {\em Proc. London Math. Soc.}, 17(2):75--115, 1918.

\bibitem[HT04]{MuPAD-Combinat}
Florent Hivert and Nicolas~M. Thi{\'e}ry.
\newblock Mu{PAD}-{C}ombinat, an open-source package for research in algebraic
  combinatorics.
\newblock {\em S\'em. Lothar. Combin.}, 51:Art. B51z, 70 pp. (electronic),
  2004.
\newblock \texttt{http://mupad-combinat.sf.net/}.

\bibitem[Kro95]{Krob.EC.1995}
Daniel Krob.
\newblock Éléments de combinatoire.
\newblock Notes de cours, Magistère 1ère année, Novembre 1995.

\bibitem[LS82]{LascouxSchutzenberger82}
Alain Lascoux and Marcel-Paul Schützenberger.
\newblock Polyn\^omes de {S}chubert.
\newblock {\em C. R. Acad. Sci. Paris S\'er. I Math.}, 294(13):447--450, 1982.

\bibitem[LS87]{LascouxSchutzenberger87}
Alain Lascoux and Marcel-Paul Schützenberger.
\newblock Symmetrization operators in polynomial rings.
\newblock {\em Functional Analysis}, 21(4):77--78, 1987.

\bibitem[Mac91]{MacDonald.1991}
I.~G. Macdonald.
\newblock Schubert polynomials.
\newblock In {\em Surveys in combinatorics, 1991 (Guildford, 1991)}, volume 166
  of {\em London Math. Soc. Lecture Note Ser.}, pages 73--99. Cambridge Univ.
  Press, Cambridge, 1991.

\bibitem[Mac95]{MacDonald.SF.1995}
I.~G. Macdonald.
\newblock {\em Symmetric functions and {H}all polynomials}.
\newblock Oxford Mathematical Monographs. The Clarendon Press Oxford University
  Press, New York, second edition, 1995.
\newblock With contributions by A. Zelevinsky, Oxford Science Publications.

\bibitem[Sag91]{Sagan.SG}
Bruce~E. Sagan.
\newblock {\em The symmetric group}.
\newblock Wadsworth \& Brooks/Cole Advanced Books \& Software, Pacific Grove,
  CA, 1991.
\newblock Representations, combinatorial algorithms, and symmetric functions.

\bibitem[Stu93]{Sturmfels.AIT}
Bernd Sturmfels.
\newblock {\em Algorithms in invariant theory}.
\newblock Springer-Verlag, Vienna, 1993.

\bibitem[{The}96]{MuPAD.96}
{The MuPAD Group, Benno Fuchssteiner et al.}
\newblock {\em MuPAD User's Manual - MuPAD Version 1.2.2}.
\newblock John Wiley and sons, Chichester, New York, first edition, march 1996.
\newblock includes a CD for Apple Macintosh and UNIX.

\bibitem[TY93]{Terasoma_Yamada.1993}
T.~Terasoma and H.~Yamada.
\newblock Higher specht polynomials for the symmetric group.
\newblock {\em Proc. Japan Acad.}, (69):41--44, 1993.

\bibitem[Woo97]{Wood.DOSA.1997}
R.~M.~W. Wood.
\newblock Differential operators and the {S}teenrod algebra.
\newblock {\em Proc. London Math. Soc. (3)}, 75(1):194--220, 1997.

\bibitem[Woo98]{Wood.PSA.1998}
R.~M.~W. Wood.
\newblock Problems in the {S}teenrod algebra.
\newblock {\em Bull. London Math. Soc.}, 30(5):449--517, 1998.

\bibitem[Woo01]{Wood.HPSA.2001}
R.~M.~W. Wood.
\newblock {H}it problems and the {S}teenrod algebra.
\newblock In {\em Proceedings of the summer school Interactions between
  Algebraic topology and invariant theory, University of Ioannina, Greece, June
  2000}. University of Ioannina reports, jun 2001.

\end{thebibliography}

\end{document}